\documentclass[12pt, 14paper,reqno]{amsart}
\usepackage{amsmath,amsfonts,amssymb}

\theoremstyle{plain}
\newtheorem{theorem}{Theorem}
\newtheorem{lemma}{Lemma}

\newtheorem{proposition}{Proposition}[section]
\theoremstyle{definition}

\theoremstyle{remark}

\numberwithin{equation}{section}
\numberwithin{lemma}{section}
\numberwithin{theorem}{section}
\usepackage{amsmath}
\usepackage{amsfonts}   
\usepackage{amssymb}
\usepackage{amssymb, amsmath, amsthm}
\usepackage[breaklinks]{hyperref}

\usepackage{graphicx}
\usepackage{amsthm}
\newtheorem{definition}{Definition}
\begin{document} 

\title[A family of elliptic curves]{
On a family of elliptic curves  of rank at least $2$}
\author{Kalyan Chakraborty and Richa Sharma}
\address{Kalyan Chakraborty @Kerala School of Mathematics, Kozhikode-673571, Kerala, India.}
\email{kalychak@ksom.res.in}
\address{Richa Sharma @Kerala School of Mathematics, Kozhikode-673571, Kerala, India.}
\email{richa@ksom.res.in}
\maketitle

\begin{abstract} 
Let $C_{m} : y^{2} = x^{3} - m^{2}x +p^{2}q^{2}$ be a family of  elliptic curves over $\mathbb{Q}$, where  $m$ is a positive integer and $p, q$ are distinct odd primes. We study the torsion part and the rank of $C_m(\mathbb{Q})$. More specifically, we  prove that  the torsion subgroup of $C_{m}(\mathbb{Q})$ is trivial and the $\mathbb{Q}$-rank of this family is at least $2$, whenever $m \not \equiv 0 \pmod 4$ and $m \equiv 2 \pmod {64}$ with neither $p$ nor $q$ divides $m$.
\end{abstract}

\section{Introduction}
The arithmetic of elliptic curves is one of the most fascinating branches in mathematics which has exciting practical applications too.
In  2002, Brown and Myers \cite{BM02} showed 
 $E_{m}:y^2=x^3-x+m^2$ has trivial torsion when $m \ge 1$, rank$(E_{m}(\mathbb{Q})) \ge 2$ if $m \ge 2$ and rank$(E_{m}(\mathbb{Q})) \ge 3$ for infinitely many values of $m$.
 Antoniewicz  \cite{AA05} considered the family $C_{m}: y^2 = x^3 - m^{2}x +1$ 
 and derived a lower bound on the rank. He showed that rank $(C_{m}(\mathbb{Q})) \ge 2$ for $m \ge 2$ and  rank $(C_{4k}(\mathbb{Q})) \ge 3$ for the infinite subfamily with $k \ge 1$. 
 Later Petra \cite{TP93} in 2012  gave a parametrization on $E:Y^2=X^3-T^2 X +1$ of rank at least $4$ over the function fields and with the help of this he found a family of rank $\ge 5$ over the field of rational functions  and a family of rank $\ge 6$ over an elliptic curve. 
 Petra again in another work \cite{TP12} considered $E:Y^2=x^3-x+T^2$, a parametrization of rank $\ge 3$ over the function fields and using this he found families of rank $ \ge 3$, $ \ge 4$ over fields of rational functions. He also  obtained a particular elliptic curve with rank $r \ge 11$. 
 More recently, Juyal and Kumar \cite{AS18} considered the family $E_{m,p}: y^2 = x^3 -m^2x+p^2$ and showed that the lower bound for the rank of  $E_{m,p}(\mathbb{Q})$ is $2$. 
 
Extending the study further, we generalize the family $E_{m,p}: y^2 = x^3 -m^2x+p^2$ by including one more prime $q$  with some conditions on the integer $m$.
 \section{Preliminaries} 
In this section we shall recall some basic facts in the theory of elliptic curves and fix the notations along the way.

Throughout this article, we denote the family of elliptic curves 
$$
y^{2} = x^{3} - m^{2}x +p^{2}q^{2}
$$ 
by $C_{m}$.
\subsection{Elliptic curve}
\begin{definition}
Let $K$ be a number field with characteristic not equal to $2$ or $3$. An algebraic curve $E$ over $K$ is defined to be an elliptic curve and is given by 
$$
E:y^{2}=x^{3}+bx+c~{\text{with}}~b,c \in K,
$$
and $\Delta=-(4b^{3}+27c^{2}) \neq 0$.
\end{definition}
It is a smooth curve which is encoded in the discriminant $\Delta \neq 0$ and this also signifies that 
$x^{3}+bx+c$ having $3$ distinct roots. This ensures that the curve is non-singular.

Let $E(K)$ denote the set of all $K$- rational points on $E$ with an additional point $\mathcal{O}$, `the point of infinity', i.e.
$$
E(K) =\left\lbrace (x,y) \in K \times K: y^{2}=x^{3}+bx+c\right\rbrace \cup \left\lbrace \mathcal{O}\right\rbrace.
$$
\begin{proposition} \cite{JJ92}
$E(K)$ forms a finitely generated abelian group under $\oplus$. The point $\mathcal{O}$ is the identity under this operation. 
\end{proposition}
The group $E(K)$ is known as the Mordell-Weil group of $E$ over $K$. The above result over $\mathbb{Q}$ is due to L.J.A Mordell and that over any number field is due to A. Weil. 

{\it  Mordell-Weil Theorem} states that 
$$
\text{E(K)}~\cong~ \text{E(K)}_{tors}~\times \mathbb{Z}^{r}.
$$
Here $E(K)_{tors}$ (the torsion part of $E$) is finite which consist of the elements of finite order on $E$  and the non-zero positive number  $r$ is called the rank of $E$ which gives us the information of how many independent points of infinite order $E$ has. It is exactly the number of copies of $\mathbb{Z}$ in the above theorem.

The structure of the torsion subgroup of an elliptic curve over $\mathbb{Q}$ is well understood. Mazur \cite{MT} and Nagell-Lutz theorems \cite{JJ92} provide a complete description of the torsion subgroup of any elliptic curve over $\mathbb{Q}$.
 The rank of an elliptic curve is a measure of the size of the set of rational points.
The rank is very difficult to compute and is quite mysterious too. There exists no known procedure which can compute the rank with sortie.
 



\section{Main result}
 Now we state the first main result of the paper.
\begin{theorem} \label{thm1}
 Let
 \begin{equation} \label{eqa}
C_{m}: y^2 = x^3 - m^2 x +  p^{2}q^{2}
\end{equation}
be a family of elliptic curves with $p, q$ are distinct primes and $m$ is a positive integer. 
Then 
$$
C_{m}(\mathbb{Q})_{tors} = \left\lbrace {\mathcal{O}}\right\rbrace
$$ 
and the $\mathbb{Q}$-rank of 
this family is at least $2$,
whenever $m \not \equiv 0 \pmod 4$ and $m \equiv 2 \pmod {64}$ with neither $p$ nor $q$ divides $m$.
\end{theorem}
Here are the main steps that are been used to prove this theorem. Firstly we use the the technique of reduction modulo a prime of an elliptic curve at good reduction primes, i.e, the primes which do not divide the discriminant of the curve. Then application of Theorem \eqref{th1} gives an injective map from the group of rational torsion points $E(\mathbb{Q})_{tors}$ into the group $E(\mathbb{F}_{p})$ to arrive onto the result on the torsion part. Further  we show that our family of concern, $C_{m}$, has at least two independent rational points, showing the rank is at least $2$.


 If $P = (x, y)$  is any point on $C_{m}$ then the law for doubling a point on an elliptic curve,
 denoted $2P = (x', y')$, is given by
\begin{equation} \label{eqaa}
x'= \frac{x^4+m^4+2m^{2}x^{2}-8p^{2}q^{2}x}{4y^{2}},
\end{equation}
$$
 y'= -y- \frac{3x^{2}-m^{2}}{2y} (x'-x) .
$$ 
The following result regarding the restriction of the reduction modulo $p$ map to the torsion part will be of our use.
\begin{theorem} \label{th1}
 ( \cite{HD87}, Theorem 5.1).
 Let $E$ be an elliptic curve over $\mathbb{Q}$. 
 The restriction of the reduction homomorphism $r_{p| E(\mathbb{Q})_{tors}} : E(\mathbb{Q})_{tors}  \to E_{p}(\mathbb{F}_{p})$ is injective for any odd prime $p$ where $E$ has a good reduction and $r_{2| E(\mathbb{Q})_{tors}} :  E(\mathbb{Q})_{tors} \to  E_{2}(\mathbb{F}_{2})$ has kernel at most $\mathbb{Z}/2\mathbb{Z}$ when $E$ has good reduction at $2$.
\end{theorem}
We begin the journey towards the proof of the above result by proving a couple of lemmas 
dealing with points of order $2, 3,5$ and $7$  of the family of elliptic curves under consideration. Here for us $m \not \equiv 0 \pmod 4$ is a positive integer and $p, q$ will always be distinct odd primes.
\begin{lemma} \label{l2}
$C_{m}(\mathbb{Q})$
doesn't have a point of order $2$.
\end{lemma}
\begin{proof}
Suppose $C_{m}(\mathbb{Q})$ has a point $A=(x,y)$  of order $2$, then $2A= \left\lbrace {\mathcal{O}}\right\rbrace \Leftrightarrow A= -A \Leftrightarrow y= 0  ~\text{and}~  x \ne 0 .$ Therefore
 $ x^3 - m^2 x + p^2 q^2 = 0.$
Since the order of $A$ is finite, so $x$ must be an integer (by Nagell-Lutz theorem \cite{JJ92}). Thus 
$$
m^2 = x^2 + \frac{p^{2}q^{2}}{x}.
$$
Which implies that 
$$ 
x \in  \{ \pm 1, \pm p, \pm p^{2}, \pm q, \pm q^{2}, \pm pq, \pm pq^{2}, \pm p^{2}q , \pm p^{2} q^{2} \}.
$$
Therefore the possible choices of $m^2$ fall in
 $$
 \{  1 \pm  p^{2} q^2, p^{2} \pm pq^2,  p^{4} \pm q^{2}, q^{2} \pm p^{2}q, q^{4} \pm p^{2}, p^{2}q^{2} \pm pq, p^{2}q^{4} \pm p, \}
 $$
 and 
 $$
 \{
 p^{4}q^{2} \pm q, p^{4} q^{4} \pm 1 \}
 $$
which is a contradiction to the fact that $m$ is an integer.  
\end{proof}
\begin{lemma} \label{l3}
$C_{m}(\mathbb{Q})$ does not contain a point of order $3$.
\end{lemma}
\begin{proof}
Suppose on the contrary that it has a point of order $3$ and call it $A$. Then $3A = \{ {\mathcal{O}} \}$, or equivalently, $2A = -A$ or $x$ - coordinate of $(2A)$ = $x$ - coordinate of $(-A)$, where $A = (x,y), 2A = (x',y')$ ($x'$ and $y'$ are given  in \eqref{eqaa}). Putting the value of $x'$ upon simplification give
\begin{equation}  \label{eq3}
m^4+6m^{2}x^{2}-3x^4-12p^{2}q^{2}x= 0.
\end{equation}
Since the order of $A$ is finite, so $x$ must be an integer (by Nagell-Lutz theorem \cite{JJ92}).

Equation \eqref{eq3} is a polynomial in  $m^{2}$ with discriminant
$$
\Delta (m) = 48x(x^{3}+p^{2}q^{2})
$$ 
and
$$
m^{2}=\frac{-6x^{2}+\sqrt {\Delta(m)}}{2}.
$$
The second relation states that 
$$
\Delta(m) = (2m^{2} + 6x^{2})^{2}
$$ 
is  square of an integer and so there exists $n \in \mathbb{N}$  such that 
\begin{equation} \label{eqn}
 x(x^{3} + p^{2}q^{2}) = 3 n^{2}. 
 \end{equation}
Now the possibilities of
$$
(x, x^{3} + p^{2}q^{2}) \in \{p,q,p^{2},q^{2},p^{2}q,pq^{2},pq ~\text{or}~ p^{2}q^{2}\}.
$$ 
Then from \eqref{eqa}, \eqref{eq3} and \eqref{eqn}, we arrive at a contradiction by $(m,p)=1, (m,q)=1$ and $(n,3)=1$.\\
Let us see the reasoning that $(n,3) =1$.\vspace*{2mm}\\
We have the relationship \eqref{eqn}

and we need to conclude that $(n,3)=1$.  Also we know that
\begin{equation} \label{aaaa}
(m^{2} + 3x^{2})^{2} = 12x(x^{3}+p^{2}q^{2}).
\end{equation}
 First we show that $x$ is odd. If possible, let $x$ is even. Then
there are  two cases to be considered:
\begin{itemize}
    \item If $x \equiv 0 \pmod{4}$ and then $x=4k, k \in \mathbb{Z}$.
Since $m \equiv 2 \pmod{64}$ we have\ $m=2 +64 k_{1}$. Putting the values of $m$ and $x$ in \eqref{eq3}, we obtain
\begin{equation} \label{qq}
    (2+64k_{1})^4+96k^{2}(2+64k_{1})^2-768k^4-48p^2q^2k=0
\end{equation}

Now reducing \eqref{qq} modulo $12$, we get
  $$
    2 + 64k_{1}  \equiv 0 \pmod{12}.
    $$
    This implies 
    $$
    2k_{1} \equiv 5 \pmod{6},
    $$
which is not possible.
 \item If $x \equiv 2 \pmod{4}$ and then $x=2+4k$. Again using 
 $m \equiv 2 \pmod{64}$ gives $m=2 +64 k_{1}$. Further reducing \eqref{aaaa} modulo $32$, we get   
 $$
    3(1+2k) \equiv 0 \pmod{4},
    $$
    which is also not possible. Thus $x$ is odd. 
    
\end{itemize}
If possible let $n \equiv 0 \pmod 3$ and in that case using \eqref{eqn} we write
\begin{equation}\label{n3}
    (2m^2+6x^2)^2=48\times 3n^2.
\end{equation}

This gives $m \equiv 0 \pmod 3$ and let $m=3m_{1}$. Further 
putting the value of $m$ in \eqref{n3} and using $n \equiv 0 \pmod 3$, we get $x \equiv 0 \pmod 3$. Let $x=3x_{1}$.

Now substituting values of $m$ and $x$ in \eqref{eq3}, we obtain
\begin{equation} \label{qqq}
    9m_{1}^4+54m_{1}^2x_{1}^2-27x_{1}^4-4p^2q^2x_{1}=0.
\end{equation}
Thus
$9 | 4p^2q^2x_{1}.
$
Which implies $9 \mid x_{1}$ (since $(m,p)=1$) and write $x_{1} = 9x_{2}$. 
Thus \eqref{qqq} transforms to (substituting $x_{1} = 9x_{2}$)

\begin{equation} \label{aaa}
    m_{1}^{4}+2(3^5)m_{1}^{2}x_{2}^{2}-3^{9}x_{2}^{4}-4p^2q^2x_{2}=0.
\end{equation}
Finally using the facts that $x$ is odd and  $m \equiv 2 \pmod{64} $ and 
reducing \eqref{aaa} modulo 8, we conclude

$$
4 x_{2} \equiv 5 \pmod{8}.
$$
This is impossible.

Hence $(x, x^{3} + p^{2}q^{2}) = 1$ and two cases may occur:
$$
y^{2}=x^{3}-m^{2}x+p^{2}q^{2}~ (\text{where}~p \ne q),
$$
\begin{tabular}{r|p{11cm}}
$
\begin{cases}
x= \alpha^{2},\\
x^{3} + p^{2}q^{2} = 3 \beta^{2}
\end{cases}
$
with $(\alpha, 3\beta)=1$
&
$\begin{cases}
x= 3\alpha^{2},\\
x^{3} + p^{2}q^{2} = \beta^{2}
\end{cases}$
with $(3\alpha, \beta)=1$
\end{tabular} \\
{\bf Case (I)}
As  $\alpha \not \equiv 0 \pmod 3$,
substituting the value of $x$ in 
$x^{3} + p^{2}q^{2} = 3 \beta^{2}$ and reducing it  modulo $3$ entails
$$
\alpha^{6} + p^{2}q^{2} \equiv 0 \pmod 3.
$$
Since 
$\alpha \not \equiv 0 \pmod 3$ this implies $\alpha^{6} \equiv 1 \pmod 3$ and we get
$$
p^{2}q^{2} \equiv 2 \pmod 3,
$$ 
which is not possible.\\
{\bf Case (II)}
Substituting $x$ in 
$x^{3} + p^{2}q^{2} =  \beta^{2}$ and reading it modulo $3$, we get
$ p^{2}q^{2} \equiv \beta^{2} \pmod 3$.
As  $\beta \not \equiv  0 \pmod 3$, this would give $\beta^{2} \equiv 1 \pmod 3$. Thus 
\begin{equation} \label{eq41}
p^{2}q^{2}  \equiv 1 \pmod 3.
\end{equation}
 Now $m$ is a multiple of $3$ (from \eqref{eq3}) and let $m=3m_{1}$. 
Since $x$ is a multiple of $3$ and $x=3 \alpha^{2}$, putting the value of $m$ and $x$ in \eqref{eq3} entails
$$
9m_{1}^{4}+36m_{1}^{2} \alpha^{4}-27 \alpha^{8}-4p^{2}q^{2} \alpha^{2} =0.
$$
This implies $3 | (pq\alpha)$ and that give rise to two possibilities: \\
{\bf (i)} If $3 | p$ then $p=3$  and \eqref{eq41} gives the contradiction.
The case when $3|q$ can be tackled analogously.\\
{\bf (ii)}  $3\not\vert\alpha$ as $(3, \alpha)=1$. 
Suppose $(3, \alpha) \ne 1$, then
$\alpha = 3^a \alpha_{1}.$
Now we are in the case when
\begin{equation} \label{a1}
x= 3\alpha^{2}.    
\end{equation}
Putting the value of $\alpha$ in \eqref{a1}, we get
$x=3^{2a+1}\alpha_{1}^{2}$.\\
Now we consider \eqref{eqn} and put the value of $x$, we get a contradiction via $(n,3) \ne 1.$

\end{proof}
Let $P  \in C_m(\mathbb{Q}) = (x,y), 2P = (x',y')$ (notations are as before) and then doubling of $2P$, i.e. 
 $4P = (x'', y'')$, where
\begin{eqnarray}\label{rs11}
x''&=& \frac{{x'}^4+m^4+2m^{2}{x'}^{2}-8p^{2}q^{2}x'}{4{y'}^{2}},\nonumber\\
 y''&=& -y'- \frac{3{x'}^{2}-m^{2}}{2y'} (x''-x').
\end{eqnarray}
\begin{lemma} \label{l5}
$C_{m}(\mathbb{Q})$ does not contain a point of order $5$.
\end{lemma}
\begin{proof}
 Suppose $C_m$ has a 
 order $5$ point $A$. Then $5A = \{ {\mathcal{O}} \}$, or equivalently  $4A = -A$ or $x$- coordinate of $(4A) =  x$- coordinate of $(-A)$, where $A = (x,y), 4A = (x'',y'')$ ($x''$ and $y''$ are given in \eqref{rs11}). Upon simplification after putting the value of $x''$ entails
\begin{eqnarray} \label{eq21}
 &&(x^{4}+m^4+2m^2x^2-8p^2q^2x)^4+256m^4y^8\nonumber\\
 &+&32 m^2y^4(x^4+m^4+2m^2x^2-8p^2q^2x)^2\nonumber\\
&-&512 p^2q^2y^6(x^4+m^4+2m^2x^2-8p^2q^2x)\nonumber\\
&=&256 xy^6 \times \nonumber\\
&[&\hspace*{-4mm}-2y^2-(3x^2-m^2)(\frac{x^4+m^4+2m^{2}x^{2}-8p^{2}q^{2}x}{4y^{2}}-x)]^2.
\end{eqnarray}
If $x$ is even and we read \eqref{eq21} modulo $4$, that would give $m \equiv 0 \pmod 4$. Which is not possible as we have assumed $m \not\equiv 0 \pmod 4$.

In the case when $x$ is odd, again reducing  \eqref{eq21} modulo $4$ will give
$$
(m^2+1)^8 \equiv 0 \pmod 4.
$$
This implies $m \equiv1,3 \pmod 4$,
which is again not possible as we have assumed $m \equiv 2 \pmod 4.$
\end{proof}
Appealing to  the addition law on $C_m$ once more, 
$$
6P = 2P \oplus  4P = (x''', y''') ~~(\text{say}),
$$ 
with 
\begin{eqnarray*}
x''' &=& \frac{(y''-y')^2}{(x''-x)^2}-x'-x''.\end{eqnarray*}
 Here $x''$ and $y''$ are given  in \eqref{rs11}.
 
We now proceed to rule out the existence of order $7$-point on $C_m$.
\begin{lemma} \label{l7}
$C_{m}(\mathbb{Q})$ does not have a point of order $7$.
\end{lemma}
\begin{proof}
Let $A=(x,y) \in C_{m}(\mathbb{Q})$ be of order $7$. Then

$7A= \left\lbrace {\mathcal{O}}\right\rbrace \Leftrightarrow 6A= -A \Leftrightarrow$ $x$-coordinate of $(6A)$ = $x$-coordinate of $(-A)$. 

Thereafter performing some elementary simplifications we arrive at
\begin{eqnarray}\label{equ}
&& 16y^{2}y'^4[4y'^{2}+(3x'^2-m^{2})(x''-x')]^2\nonumber\\
&-&(x'^4+m^4+2m^2x'^2-8p^2q^2x'-4xy'^2)^2 \nonumber\\ 
&& [y'^2 (x^4+m^4  +2m^2x^2-8p^2q^2x) +  y^2(x'^4+m^4+2m^2x'^2-8p^2q^2x')]\nonumber\\ 
&=& 4xy^2y'^2(x'^4+m^4+2m^2x'^2-8p^2q^2x'-4xy'^2)^2.
\end{eqnarray}
We reduce \eqref{equ} modulo $4$ to get
\begin{eqnarray} \label{eq4}
 &-&(x'^4+m^4+2m^2x'^2)^2 [y'^2 (x^4+m^4  +2m^2x^2) +  y^2(x'^4+m^4+2m^2x'^2)]\nonumber\\
 &\equiv& 0 ~ \pmod 4.
\end{eqnarray}
Now  two cases can occur:\\
{\bf Case I.}  If $x \equiv 0 \pmod 2.$ 
In this case, 
$$
m \equiv 0 \pmod 4
$$
and that  is not possible because by assumption $m \not \equiv 0 \pmod{4}$.\\
{\bf Case II.}  If $x \not \equiv 0 \pmod 2$. 


In this case $x^2 \equiv 1 \pmod{8}$ as $x$ is odd.
Now reducing \eqref{equ} modulo $8$, we obtain
\begin{eqnarray}\label{equ1}
&&
-(x'^4+m^4+2m^2x'^2)^2 \nonumber\\ 
&& [y'^2 (1+m^4  +2m^2) +  y^2(x'^4+m^4+2m^2x'^2)]\nonumber\\ 
&=& 4xy^2y'^2(x'^4+m^4+2m^2x'^2)^2 \pmod{8}.
\end{eqnarray}
Further from \eqref{eqaa} we see that
$$
x'= \frac{1+m^4+2m^2}{4y^2} \pmod{8}
$$
$$
x'^{2}= \frac{(1+m^4)^2+4m^4+4m^2(1+m^4)}{16y^4} \pmod {8}
$$
$$
x'^{4}= \frac{(1+m^4)^4}{16^2y^8} \pmod{8}
$$
$$
y'=-\frac{1}{8y^3}(3-m^2)(m^4+6m^2-4x-3) \pmod{8}
$$
$$
y'^{2}= \frac{1}{64y^6}(3-m^2)^2(1+m^4+2m^2)^2 \pmod{8}.
$$
Now substituting values of $x', x'^2, x'^4, y'$ and $y'^2$ in \eqref{equ1}, we get 
\begin{equation}\label{equ11}
   (1+m^4)^8[4(3-m^2)^2(1+m^4+2m^2)^3+(1+m^4)^4] \equiv 0 \pmod{8}.
\end{equation}
Now using $m \equiv 2 \pmod{64}$ in \eqref{equ11} gives $5 \equiv 0 \pmod{8}$, which is not possible.

\end{proof}
\section{Proof of Theorem \ref{thm1}}
We are now in a position to complete the proof of theorem \ref{thm1}.
\begin{proof}
Before proceeding towards the proof let us recall that $3 \mod p$ is not a square in $({\mathbb{Z}} / {p \mathbb{Z}})^*$ for $p=5,7$ and $17.$\\
The discriminant of $C_m$
$$
\Delta  (C_{m} ) = 16(4m^{6}-3^{3}p^{4}q^{4}).
$$
{(I)}
If $p,q \ne 5$ and $m \not \equiv 0 \pmod 4$ then  $5\not\vert \Delta (C_{m})$ and thus $C_{m}$ has good reduction at $5$. Now two cases may occur while reducing $C_{m}$
to $\mathbb{F}_{5}$.
\begin{itemize}
\item[(a)]  If $p^{2} \equiv 1 \pmod 5$ then
$q$ is $q^{2} \equiv 1,4 \pmod 5$ and  that forces $p^{2}q^{2} \equiv 1,4 \pmod 5$.
\begin{itemize}
\item[(i)] 
When $p^{2}q^{2} \equiv 1 \pmod 5$,
 the curve $C_{m}$ reduces to $y^{2} =x^{3}+1$, $y^{2} =x^{3}-x+1$ and $y^{2}=x^{3}-4x+1$ according as $m^{2} \equiv  0,1 ~\text{or}~ 4 \pmod 5$ respectively. The corresponding size of $C_{m}(\mathbb{F}_{5})$ would be  $6,8$ and $9$.
 \item[(ii)]
When $p^{2}q^{2} \equiv 4 \pmod 5$, depending upon whether $m^{2} \equiv  0,1 ~\text{or}~ 4 \pmod 5$,
 the curve $C_{m}$ reduces to $y^{2} =x^{3}+4$, $y^{2} =x^{3}-x+4$ and $y^{2}=x^{3}-4x+4$  respectively with the corresponding cardinality of $C_{m}(\mathbb{F}_{5})$ being $6,8$ and $9$.
\end{itemize} 
\item[(b)]  If $p^{2} \equiv 4 \pmod 5$. This case is analogous to the previous one.
\end{itemize}
Theorem \ref{th1} and Lagrange's theorem tell us that the possible orders of  $C_{m}(\mathbb{Q})_{tors}$ are $1, 2, 3, 4 ,6,8$ and $9$ only. Lemmas \ref{l2} and \ref{l3} show that $C_{m}(\mathbb{Q})$ does not have points of order $2$ and $3$. Thus in this case
 $$
  C_{m}(\mathbb{Q})_{tors}=  \{ {\mathcal{O}}\}.
 $$
{(II)} If $p=5$ or $q=5$ (as $p \ne q$). Let $p=5$, then 
the defining equation of $C_m$ is 
$$
y^{2}=x^{3}-m^{2}x+25q^{2},
$$
and
$$
\Delta  (C_{m} ) = 16(4m^{6}-3^{3}5^{4}q^{4})
$$
respectively.
\begin{itemize}
\item[(a')] If $q \ne 7$, then utilising the condition $m \not \equiv 0 \pmod 4$ would imply $7 \not\vert  \Delta (C_{m})$, 
and thus $C_{m}$ has a good reduction at $7$. Now three cases may occur while reducing $C_{m}$
to $\mathbb{F}_{7}$.
\begin{itemize}
\item[(i)]  If $q^{2} \equiv 1 \pmod 7$: \\
depending on $m^{2} \equiv  0,1,2 ~\text{or}~ 4 \pmod 7$, the curve $C_{m}$ reduces to $y^{2} =x^{3}+4$, $y^{2} =x^{3}-x+4$, $y^{2} =x^{3}-4x+4$ and $y^{2}=x^{3}-2x+4$ with the corresponding cardinality of $C_{m}(\mathbb{F}_{7})$ would be $3,10,10$ and $10$ respectively.
\item[(ii)] If $q^{2}  \equiv 2 \pmod 7$:\\ then this case $C_{m}$ reduces to $y^{2}=x^{3}+1$,$y^{2}=x^{3}-x+1$, $y^{2}=x^{3}-2x+1$  and $y^2=x^{3}-4x+1$ according as $m^{2}  \equiv 0, 1,2 ~\text{or}~ 4 \pmod 7$ with the corresponding cardinality of $C_{m}(\mathbb{F}_{7})$  would be $12,12,12$ and $12$ respectively.
\item[(iii)]  If $q^{2} \equiv 4 \pmod 7$: \\
depending upon whether $m^{2} \equiv  0,1,2 ~\text{or}~ 4 \pmod 7$, $C_{m}$ reduces to $y^{2} =x^{3}+2$, $y^{2} =x^{3}-x+2$, $y^{2} =x^{3}-2x+2$ and $y^{2}=x^{3}-4x+2$ with the corresponding cardinality of $C_{m}(\mathbb{F}_{7})$ would be $9,9,9$ and $9$ respectively.
\end{itemize}
Thus the possible orders of $C_{m}{(\mathbb{Q})}_{tors}$ are $1, 2, 3, 4,5 ,6,9,10$ and $12$. The results of lemmas \ref{l2}, \ref{l3} and \ref{l5} show that  $C_{m}$ does not have points of order $2,3$ and $5$. Thus in this case too
 $$
  C_{m}(\mathbb{Q})_{tors}=  \{ {\mathcal{O}}\}.
 $$
\item[(b')] If $q=7$. In this case
since  $17  \nmid \Delta$, thus  the curve $C_{m}$ has a good reductions at  $17$. Now reducing $C_{m}$ to
$\mathbb{F}_{17}$, the curve
$C_m$ has various possibilities:\\  
$C_{m} :y^{2} \equiv x^{3}+1$,  $y^{2} \equiv x^{3}-x+1$,  $y^{2} \equiv x^{3}-2x+1$,  $y^{2} \equiv x^{3}-4x+1$,  $y^{2} \equiv x^{3}-8x+1$,  $y^{2} \equiv x^{3}-9x+1$,  $y^{2} \equiv x^{3}-13x+1$,  $y^{2} \equiv x^{3}-15x+1$ or  $y^{2} \equiv x^{3}-16x+1$ according as $m^{2} \equiv 0, 1, 2, 4,8,9,13,15 ~{\text{or}}~16 \pmod{17}$ respectively, with the cardinality of $C_{m}(\mathbb{F}_{17})$  would be  $18,14,16,25,21,19,24,24$ or $18$ respectively. 

Hence the possible order of $C_{m}(\mathbb{Q})_{tors}$ are 
$$
1, 2, 3,4, 5,6,7,8,9,12,14,16,18,19,21,24~ \text{or}~25.
$$
(Mazur's theorem \cite{MT}
tells that $14,18,19,21, 24 ~\text{and}~25$ can not be a possible order.)
 Among the rest  only probable value is $1$  because of lemmas \ref{l2} \ref{l3}, \ref{l5} and \ref{l7}. Thus$$
C_{m}(\mathbb{Q})_{tors} = \left\lbrace {\mathcal{O}}\right\rbrace.
$$
\end{itemize}
\end{proof}
\section{The rank of $C_{m}$}
The rank of an elliptic curve is a major topic of research for  many years now but it is  
yet to be understood well.
In this section we show that our family of concern, $C_m$, has at least
two independent rational points, showing the rank is at least $2$. These two points are in fact $A_{m}= (0, pq)$ and $B_{m} = (m, pq)$ 
are in $C_{m} (\mathbb{Q}).$
We need to show that $A_m$ and $B_m$ are linearly independent, i.e.
there do not exist non zero integers $a$ and $b$ such that 
$$
[a]A_m + [b]B_m = \mathcal{O},
$$
here $[a]A_m$ denotes $a$-times addition of $A_m$.

Points of order $2$ satisfy $y=0$, while points of order $4$ satisfy $x=0$, so for any rational point $(x,y)$ on $C_m$ such that $xy \ne 0$ must be of infinite order. Therefore in our case rank of $C_{m}$ must be least $1$.
To show that rank is $2$ we need to recall the following result.
\begin{theorem} \label{th2} (\cite{CJ97}). Let $E(\mathbb{Q})$ (respectively $2E(\mathbb{Q})$) be the group of rational points (respectively, double of rational points) on an elliptic curve $E$, and suppose that $E$ has trivial rational torsion. Then the quotient group $\frac{E(\mathbb{Q})}{2E(\mathbb{Q})}$ is an elementary abelian $2$-group of order $2^{r}$, where $r$ is the rank of  $E(\mathbb{Q})$.
\end{theorem}
\begin{lemma} \label{le1}
 Let $A = (x', y')$ and $B = (x, y)$ be points in $C_{m}(\mathbb{Q})$ such that $A=2B$ and $x' \in  \mathbb{Z}$. Then
 \begin{itemize}
 \item $x \in  \mathbb{Z},$
\item $x \equiv m \pmod 2.$
\end{itemize}
\end{lemma}
\begin{proof}
Substituting $x= \frac{u}{s}$ with $(u, s)=1$ in \eqref{eqaa} and after elementary simplification
$$
u^{4}-4x'u^{3}s+2m^{2}u^{2}s^{2}+(4m^{2}x'-8p^{2}q^{2})us^{3}+(m^{4}-4p^{2}q^{2}x')s^{4}=0.
$$
This relation implies that  $s \vert u^4$, and therefore $s=1$.
Thus $x \in \mathbb{Z}$.\\
Again,  \eqref{eqaa} can be written as
$$
(x^{2}+m^{2})^{2}=4[x'x^{3}-m^{2}x+25p^{2}+50p^{2}x],
$$
which implies that $2 \mid (x^{2} + m^{2})$. Thus $x \equiv m \pmod 2$.
\end{proof}
\begin{lemma} \label{la}
The equivalence class $[A_{m}]=[(0,pq)]$ is a non-zero element of ${C_{m}(\mathbb{Q})}/{2C_{m}(\mathbb{Q})}$ for any positive integers $m$ with $m \equiv 2 \pmod {64}$ and for odd primes $p$ and $q$.
\end{lemma}
\begin{proof}
Assume  $A_{m}=2C$ for some $C=(x,y) \in C_{m}(\mathbb{Q})$. Thus 
$$
\frac{x^4+m^4+2m^{2}x^{2}-8p^{2}q^{2}x}{4y^{2}}=0.
$$
Upon simplification it becomes
\begin{equation} \label{eq1}
x^4+m^4+2m^{2}x^{2}-8p^{2}q^{2}x=0.
\end{equation}
Thus 
\begin{equation}\label{rs22}
(x^2+m^2)^{2}=8 p^{2}q^{2} x.
\end{equation}
The left hand side of \eqref{rs22} is a square  and so the right hand side would also be a square. This implies $x = 2(k)^{2}$ for some $k \in  \mathbb{Z}$,  where $(2,k)=1$.

[ Proof of the fact that $(2,k)=1$: Suppose $(2,k) \ne 1$, then $k=2k_{1}$. Now substituting the value of $x=8k_{1}^2$ in \eqref{rs22}, we obtain 
\begin{equation} \label{k2}
  8^4k_{1}^{8} + 16m_{1}^{4} + 128k_{1}^{4}m^{2} = 64k_{1}^{2}p^{2}q^{2}.  
\end{equation}
Equation \eqref{k2} implies $m$ is a multiple of $4$ which is a contradiction to the fact that $m \not \equiv 0 \pmod 4$ ].

Putting the value of $x$ in equation \eqref{eq1}, we obtain
\begin{equation}\label{rs33}
16k^{8} + m^{4} + 8k^{4}m^{2} - 16k^{2}p^{2}q^{2}=0.
\end{equation}
When $m \equiv 2 \pmod{64}$ this implies $m^{2}  \equiv 4 \pmod {64}$ and $m^{4}  \equiv 16 \pmod{64}$.
We read \eqref{rs33}  modulo $64$ to get
$$
k^{8}+1+2k^{4}-k^{2}p^{2}q^{2} \equiv 0 \pmod 4.
$$
Since $k$ is odd and $p,q$ are odd primes so $p^{2} \equiv 1 \pmod 4$ and $q^{2} \equiv 1 \pmod 4$. Using this we get a contradiction that $3 \equiv 0 \pmod 4$.
Therefore the above equation has no solution under modulo $64$. Hence this equation  has no solution. Therefore $A_{m} \not \in 2C_{m}(\mathbb{Q}).$
\end{proof}
\begin{lemma} \label{lb}
The equivalence class  $[B_{m}]=[(m,pq)]$ is a non-zero element of  ${C_{m}(\mathbb{Q})}/{2C_{m}(\mathbb{Q})}$ for  positive integers $m \equiv 2 \pmod {4}$ and for odd primes $p, q$.
\end{lemma}
\begin{proof}
Assume $B_{m} = (m, pq) = 2C$ for some  $C=(x, y) \in  C_{m}(\mathbb{Q})$. Thus we get
$$
\frac{x^4+m^4+2m^{2}x^{2}-8p^{2}q^{2}x}{4y^{2}}=m.
$$
Since $x \equiv m \pmod 2$ (using lemma \ref{le1}), we can write $x-m=2s$ and 
after simplifying we get
$$
(x-m)^{4}-4m^{2}(x-m)^{2}-8p^{2}q^{2}(x-m)-12mp^{2}q^{2}+4m^{4}=0.
$$
Now using $x-m=2s$, we get
$$
(2s^{2}-m^{2})^{2}=p^{2}q^{2}(4s+3m).
$$
Since left side is a square this implies right would also be square so  $(4s+3m) = w^{2}$ for some $w \in  \mathbb{Z}$.\\
Since $m \equiv  2 \pmod 4$, this implies
$4s +3m \equiv 2 \pmod 4$. So we get into a contradiction by $2 \equiv 0~\text{or}~1 \pmod 4$ .
\end{proof}
\begin{lemma} \label{lab}
The The equivalence class  $[A_{m}+B_{m}]=[(-m,-pq)]$ is a non-zero element of  ${C_{m}(\mathbb{Q})}/{2C_{m}(\mathbb{Q})}$ for  positive integers $m \equiv 2 \pmod {16}$ and  odd primes $p, q$.
\end{lemma}
\begin{proof}
Suppose $A_{m}+B_{m} = (-m, -pq) = 2C$ for some  $C=(x, y) \in  C_{m}(\mathbb{Q})$. Thus 
$$
\frac{x^4+m^4+2m^{2}x^{2}-8p^{2}q^{2}x}{4y^{2}}=-m.
$$
As $x \equiv m \pmod 2$, we can write $x-m=2s$ and 
after simplifying, 
\begin{eqnarray*}
(x-m)^{4}+8m(x-m)^{3}&+&20m^{2}(x-m)^{2}+16m^{3}(x-m)\\
&-&8p^{2}q^{2}(x-m)
+4m^{4}-8p^{2}q^{2}m+4mp^{2}q^{2}\\
&=&0.
\end{eqnarray*}
Now using $x-m=2s$,
\begin{equation}\label{rs44}
4s^{4}+16ms^{3}+20m^{2}s^{2}+8m^{3}s-4p^{2}q^{2}s+m^{4}-p^{2}q^{2}m=0.
\end{equation}
When $m \equiv 2 \pmod{16}$ this implies $m^{2}  \equiv 4 \pmod {16}$ and $m^{4}  \equiv 0 \pmod{16}$.
Now reducing \eqref{rs44} modulo $16$ gives
$$
4s^{4}-4p^{2}q^{2}s-2p^{2}q^{2}  \equiv 0 \pmod {16}.
$$
This in turn gives that $2s^{4}-2p^{2}q^{2}s-p^{2}q^{2}  \equiv 0 \pmod {8}$.
Since $p$ and $q$ are odd primes so $p^{2} \equiv 1 \pmod 8$ and $q^{2} \equiv 1 \pmod 8$ and using this we arrive into a contradiction.
\end{proof}
If we show that ${\left\lbrace [{\mathcal{O}], [A_{m}], [B_{m}],[A_{m}]+[B_{m}]}\right\rbrace}$  is a subgroup of $\frac{C_{m}}{2C_{m}}$
and  $A_{m},B_{m}$ are linearly independent points then the proof that rank of $C_m(\mathbb{Q}) \ge 2$ will be completed.
 \begin{theorem}
 Let  m is a positive integer such that $m \not \equiv 0 \pmod 4$ and 
 $m \equiv 2 \pmod{64}$ with $p, q$ be odd primes then the set 
 $$
 {\left\lbrace [{\mathcal{O}], [A_{m}], [B_{m}],[A_{m}]+[B_{m}]}\right\rbrace}
 $$  is a subgroup of $\frac{C_{m}}{2C_{m}}$ of order $4$
 with $A_{m} = (0,pq), B_{m} = (m, pq)$.
 \end{theorem}
\begin{proof}
From the above  we know that $[A_{m}] \ne \left\lbrace {\mathcal{O}}\right\rbrace$, $[B_{m}] \ne  \left\lbrace {\mathcal{O}}\right\rbrace$ and $[A_{m} +  B_{m}] \ne  \left\lbrace {\mathcal{O}}\right\rbrace$.
We now assume $[A_{m}] = [B_{m}]$, then  $[A_{m}+ B_{m}] = [A_{m}] + [B_{m}] = [2A_{m}] = [\left\lbrace {\mathcal{O}}\right\rbrace]$, which is not possible. It is easy to show that 
$[A_{m}]$ and  $ [A_{m} + B_{m}]$ are distinct.  Similarly $[B_{m}]$ and $[A_{m} + B_{m}]$ are also distinct.
Hence $ [{\mathcal{O}}], [A_{m}], [B_{m}]$ and $[A_{m}]+[B_{m}]$ ] are distinct classes in $\frac{C_{m}}{2C_{m}}$.
Thus this set is a subgroup of order $4$ in $\frac{C_{m}}{2C_{m}}$.
\end{proof}
\begin{theorem}
Points $A_{m}$ and $B_{m}$ are linearly independent in $C_{m}:y^2 = x^3 - m^2 x + p^2 q^2$ with $A_{m} = (0,pq), B_{m} = (m, pq)$ and m is a positive integer such that $m \not \equiv 0 \pmod 4$ and 
 $m \equiv 2 \pmod{64}$ with $p, q$ be two odd primes.
 \end{theorem}
\begin{proof}
Assume, on the contrary $a A_{m} + b B_{m} = {\mathcal{O}}$, where $a$ and $b$ are integers with $a$ is minimal.
Four cases needed to be considered.
\begin{itemize}
\item If $a$ is even and $b$ is odd, then $[aA_{m} + bB_{m}] = [{\mathcal{O}}]$  then in the group $\frac{C_{m}}{2C_{m}}$, we get $[B_{m}]=[{\mathcal{O}}]$. Thus we get into contradiction by lemma \ref{lb}.
\item If $a$ is odd and $b$ is even, then $[aA_{m} + bB_{m}] = [{\mathcal{O}}]$  implies that $[A_{m}]=[{\mathcal{O}}]$ which is not possible by lemma \ref{la}.
\item When both $a$ and $b$ are odd, we get $[A_{m} + B_{m}]={\mathcal{O}}$, which contradicts  lemma \ref{lab}.
\item If $a$ and $b$ are both even, then writing $a = 2a',~ b = 2b'$, we get
 $$
 2[a'A_{m} + b'B_{m}]=[{\mathcal{O}}].
 $$
 Which implies that $[a'A_{m} +b'B_{m}]$ is a point of order $2$. From  lemma \ref{l2}, we get $[a' A_{m} +b' B_{m}] =[{\mathcal{O}}]$ and  this contradicts the fact that $a$ is minimal.
\end{itemize}
\end{proof}
Thus we have proved that $A_{m}$ and $B_{m}$ are linearly independent points and $\frac{C_{m}}{2C_{m}}$ contains a subgroup of order $4$.  Now by theorem \ref{th2}, the rank $r$ of $C_{m}(\mathbb{Q})$ is at least $2$ for any positive integer $m$, with $m \not \equiv 0 \pmod 4, m \equiv  2 \pmod {64}$ and $p, q $ be two odd primes.
 \begin{table}[h!]
 \vspace{-.3cm}
\caption{Rank of $C_{m}(\mathbb{Q})
:y^2 = x^3 - m^2 x + p^2q^2$  for some values of $m,p$ and $q$.} \label{T1} 
\resizebox{\columnwidth}{!}{
 \centering
\begin{tabular}{|c|c|}
\hline
{\textbf{Rank}} & {\textbf{$[m,pq]$}} \\ 
\hline
$ 2$ &  $(2, 21)(2, 33)(194, 33)(2, 39)(194, 51)(2, 57)(194, 57)(258, 35)(2, 55)(2, 65)(66, 65),  $ \\ 
        &  $ (258, 65)(2, 85)(2, 95)(2, 21)(258, 35)(258, 77)(66, 91)(194, 91)(258, 91)(130, 119),  $ \\ 
        &  $(2, 133)(130, 133)(258, 133)(2, 33)(194, 33)(2, 55)(258, 77)(130, 187)(258, 187)(130, 209),  $ \\ 
        &  $(194, 209)(258, 209)(2, 39)(2, 65)(66, 65)(258, 65)(66, 91)(194, 91)(258, 91), $ \\ 
        &  $(258, 221)(2, 247)(66, 247)(194, 247)(194, 51)(2, 85)(130, 119)(130, 187)(258, 187), $ \\ 
        &  $(258, 221)(2, 57)(194, 57)(2, 95)(2, 133)(130, 133)(258, 133)(130, 209)(194, 209)(258, 209) $ \\ 
        &  $(2, 247)(66, 247)(194, 247) $ \\ 
         \hline
          $3$ &  $ (2, 15)(130, 21)(194, 21)(194, 39)(2, 51)(130, 51)(2, 15)(2, 35)(66, 35)(194, 65)(66, 85),$\\ 
        & $ (194, 85)(66, 95)(194, 95)(130, 21)(194, 21)(2, 35)(66, 35)(130, 77)(194, 77)(2, 91), $\\ 
        & $(66, 119)(258, 119)(194, 133)(130, 77)(194, 77)(2, 143)(258, 143)(2, 187)(2, 209)(194, 39), $\\ 
        & $(194, 65)(2, 91)(2, 143)(258, 143)(2, 221)(258, 247)(2, 51)(130, 51)(66, 85)(194, 85), $\\ 
        & $(66, 119)(258, 119)(2, 187)(2, 221)(130, 323)(258, 323)(66, 95)(194, 95)(194, 133)(2, 209), $\\ 
        & $(258, 247)(130, 323)(258, 323) $\\ 
 \hline
 $4$ & $ (194, 15)(130, 33)(130, 57)(194, 15)(194, 35)(194, 55)(258, 55)(258, 85)(194, 35)(2, 77), $\\ 
        & $ (2, 119)(194, 119)(66, 133)(130, 33)(194, 55)(258, 55)(2, 77)(194, 143)(194, 143), $\\ 
        & $(66, 221)(194, 221)(258, 85)(2, 119)(194, 119)(66, 221)(194, 221)(66, 323)(194, 323), $\\ 
        & $ (130, 57)(66, 133)(66, 323)(194, 323)$\\ 
         \hline
$5$ & $(258, 95)(194, 187)(194, 187)(2, 323)(258, 95)(2, 323)$\\ 
      \hline
\end{tabular}
 }
\end{table}
Table \ref{T1} confirms the  results up to certain values of $m,p$ and $q$. The table shows that  the rank of the elliptic curves considered is $\ge 2$.

All the computations are done with the help of SAGE \cite{Ss45}.
\section{Concluding remarks} If we consider a larger family of elliptic curves (say) 
$$
D_{m} : y^2 = x^3 - m^2x + (p q r)^2.
$$
Then following our method the number of cases that have to be dealt with  becomes quite large and to handle that many cases will be cumbersome to say the least. 
Another interesting phenomena that we observed while computing the ranks, that many curves come out with rank $3$. A natural query would be:
 Do there exist a subfamily consisting of infinitely many members amongst the members in $C_{m}$ with rank at least $3$? Now proving that a certain set of three points is independent amounts to showing that $7$ distinct points are not doubles of rational points.
 



\begin{thebibliography}{25}

\bibitem {AA05} Antoniewicz, A., On a family of elliptic curves, {\it Universitatis Iagellonicae Acta Mathematica}, {\bf 1285}, (2005), 21--32.

\bibitem {BM02} Brown, E. and Myers, Bruce T, Elliptic curves from Mordell to Diophantus and back, {\it The American mathematical monthly}, {\bf 109(7)}, (2002), 639--649.

\bibitem {CJ97} Cremona, John E, Algorithms for Modular Elliptic Curves, 2nd ed., {\it New York: Cambridge University Press}, (1997).

\bibitem {HD87} Husem{\"o}ller, D., Elliptic Curves, (\it New York: Springer-Verlag), (1987).

\bibitem {AS18}  Juyal, A. and Kumar, Shiv D., On the family of elliptic curves $y^2= x^3-m^2 x+ p^2$, {\it Proceedings-Mathematical Sciences}, {\bf 128(5)}, (2018), 1--11.

\bibitem {MT}  Mazur B, Modular curves and the Eisenstein ideal, {\it Publ. Math. I.H.E.S.}, {\bf 47}, (1977), 33-–186.

\bibitem {JJ92} Silverman, Joseph H. and Tate, John T., Rational points on elliptic curves. {\it Undergraduate Texts in Mathematics, Springer-Verlag, New York}, (1992).

\bibitem {Ss45}  SAGE software, version 4.5.3, \url
{http://www.sagemath.org}.

\bibitem {TP12} Tadi{\'c}, P., The rank of certain subfamilies of the elliptic curve $Y^2= X^3- X+ T^2$, {\it Annales Mathematicae et Informaticae}, {\bf 40}, (2012), 145--153.

\bibitem {TP93} Tadi{\'c}, P., On the family of elliptic curve $Y^2 = X^3 -T^2 X
+1$, {\it Glasnik Matematicki}, {\bf 47(67)},(2012) 81--93.
\end{thebibliography}
\end{document}